\documentclass[12pt]{amsart}
\usepackage{bm}
 \usepackage{graphicx}

 \newtheorem{thm}{Theorem}[section]
 
 \newtheorem{lem}[thm]{Lemma}
 
 \theoremstyle{definition}
 
 \theoremstyle{remark}
 
 \numberwithin{equation}{section}

 \newcommand{\Real}{\mathbb{R}}

 \newcommand{\di}{\mathbf{div}}

 \newcommand{\tr}{\textbf{tr}}
 \newcommand{\dis}{\textbf{d}}
\newcommand{\ntorus}{\mathbb T^n}

\begin{document}

\title[On the JKO scheme]{On the Jordan-Kinderlehrer-Otto scheme}

\author{Paul W.Y. Lee}
\email{wylee@math.cuhk.edu.hk}
\address{Room 216, Lady Shaw Building, The Chinese University of Hong Kong, Shatin, Hong Kong}

\date{\today}

\begin{abstract}
In this paper, we prove that the Jordan-Kinderlehrer-Otto scheme for a family of linear parabolic equations on the flat torus converges uniformly in space.
\end{abstract}

\maketitle

\section{Introduction}

Diffusion equations including the heat equation, the Fokker-Planck equation, and the porous medium equation are gradient flows in the Wasserstein space, the space of all probability measures, with respect to the $L^2$-Wasserstein distance from the theory of optimal transportation. One way to make the above statement precise is to use a time-discretization scheme introduced in \cite{JoKiOt} which is now called the Jordan-Kinderlehrer-Otto (JKO) scheme. An interesting and insightful formal Riemannian structure of the $L^2$-Wasserstein distance was also found in \cite{Ot}. It was also shown in \cite{AmGiSa} that the above equations are examples of generalized gradient flows on abstract metric spaces (see \cite{AmGi} for a quick overview and \cite{Vi1,Vi2} for further details). There are also lots of recent activities in this area (see, for instance, \cite{BlCa,DeMeSaVe}).

We consider, in this paper, the JKO scheme for the equation
\[
\partial_t \phi_t=\Delta \phi_t+\left<\nabla\phi_t,\nabla\Psi\right>+f\phi_t
\]
on the flat torus $\ntorus$, assuming that we have a positive solution $v_0$ of the equation
\begin{equation}\label{v0}
\Delta v_0=\left<\nabla v_0,\nabla\Psi\right>+(\Delta\Psi -f)v_0,
\end{equation}
where $\Psi$ and $f$ are smooth functions on the flat torus $\ntorus$.

For this, let $\mu_0$ and $\mu_1$ be two Borel probability measure on $\ntorus$. The $L^2$ Wasserstein distance $\dis$ between $\mu_0$ and $\mu_1$ is defined as follows
\begin{equation}\label{OT}
\dis^2(\mu_0,\mu_1)=\inf_{\varphi_*\mu_0=\mu_1}\int_Md^2(x,\varphi(x))d\mu_0(x),
\end{equation}
where the infimum is taken over all Borel maps $\varphi:\ntorus\to \ntorus$ pushing $\mu_0$ forward to $\mu_1$. Minimizers of (\ref{OT}) are called optimal maps.

Let $dx^n$ be the Lebesgue measure on $\ntorus$, let $\mu=v_0dx^n$, let $\rho_0\mu$ be a probability measure $\mathbb T^n$, let $\rho_0^N=\rho_0$, and let $K>0$ be a constant. The following minimization problem has a unique solution $\rho_k^N$ for each $k=1,...,N$.
\begin{equation}\label{JKOmin}
\inf\left[\frac{1}{2}\dis^2(\rho_{k-1}^N\mu,\rho\mu)+\frac{K}{N}\int_{\ntorus}(\log \rho-\log v_0+\Psi) \rho d\mu\right],
\end{equation}
where the infimum is taken over the set of $L^1$ functions $\rho:\ntorus\to [0,\infty)$ satisfying $\int\rho d\mu=1$.

This defines, for each positive integer $N$, a sequence of functions $\{\rho_k^N|k=0,1,...\}$. This discrete scheme is the so-called JKO scheme.

Let $\phi_t^N:\left[0,K\right]\times \ntorus\to [0,\infty)$ be defined by
\[
\phi_t^{N}=\rho_{k}^N
\]
if $t$ is in $\left[\frac{kK}{N},\frac{(k+1)K}{N}\right)$ and $k=0,...,N-1$.

It follows as in \cite{JoKiOt} that a subsequence of $\{\phi_t^N| N=1,2,...\}$ converges in $L^1$ to the solution of the initial value problem
\begin{equation}\label{equa}
\partial_t \phi_t=\Delta \phi_t+\left<\nabla\phi_t,\nabla\Psi\right>+f\phi_t,\quad \phi_0=\rho_0.
\end{equation}

Note that a uniform $L^\infty$ bound for the sequence $\phi_t^N$ can be found in \cite{Sa}. In the case of the flat torus $\ntorus$, we show that this sequence has a uniform Lipschitz bound in space. More precisely,

\begin{thm}\label{main}
Assume that $\rho_0$ is in $C^4(\ntorus)$. Then there is a constant $C>0$ depending only on $\rho_0$, $v_0$, and $\Psi$ such that the sequence $\{\phi_t^N| N=1,2,...\}$ satisfies
\[
|\nabla\phi_t^N|\leq C
\]
for all $t$ in the interval $[0,K]$.

Moreover, this sequence converges in $C^{0,\alpha}(\ntorus)$, for any $0<\alpha< 1$, to the unique solution of the initial value problem (\ref{equa}) as $k\to\infty$.
\end{thm}

The rest of the sections are devoted to the proof of Theorem \ref{main}.

\section*{Acknowledgment}

The author would like to thank the referee for suggesting several important improvements and pointing out possible research directions.

\smallskip

\section{The optimal Transportation problem and the JKO Scheme}\label{OTP}

In this section, we recall basic results in the theory of optimal transportation and discuss the JKO scheme. Let $\ntorus$ be the n-dimensional torus equipped with the flat distance $d$. Let $\mu_0$ and $\mu_1$ be two Borel probability measures on $\ntorus$. Recall that the $L^2$ Wasserstein distance $\dis$ between $\mu_0$ and $\mu_1$ is the following minimization problem
\begin{equation}\label{OTre}
\dis^2(\mu_0,\mu_1)=\inf_{\varphi_*\mu_0=\mu_1}\int_Md^2(x,\varphi(x))d\mu_0(x),
\end{equation}
where the infimum is taken over all Borel maps $\varphi:\ntorus\to \ntorus$ pushing $\mu_0$ forward to $\mu_1$. Minimizers of (\ref{OTre}) are called optimal maps.

The minimization problem (\ref{OTre}) admits the following dual problem
\begin{equation}\label{dual}
\sup_{f(x)+g(y)\leq \frac{1}{2}d^2(x,y)}\left[\int_{\ntorus}fd\mu_0+\int_{\ntorus}gd\mu_1\right],
\end{equation}
where the supremum is taken over all pairs $(f,g)$ of continuous functions satisfying $f(x)+g(y)\leq\frac{1}{2}d^2(x,y)$ for all $x$ and $y$ in $\ntorus$.

The maximizers of the above problem are given by pairs of $c$-concave functions. If $f:\ntorus\to\Real$ is a continuous function, then the $c$-transform $f^c$ of $f$ is defined by
\[
f^c(x)=\inf_{y\in\ntorus}\left[\frac{1}{2}d^2(x,y)-f(y)\right].
\]
The function $f$ is $c$-concave if $f^{cc}=f$.

For the proof of following result (see, for instance, \cite{Br,Co,Mc}).
\begin{thm}\label{dualexist}
The infimum in (\ref{OT}) and the supremum in (\ref{dual}) coincide. Moreover, the supremum in (\ref{dual}) is achieved by a pair $(f,f^c)$, where $f$ is a $c$-concave function.
\end{thm}

The following existence and uniqueness result can be found in \cite{Br,Mc}. Note that $c$-concave functions are locally semi-concave and hence twice differentiable almost everywhere (see \cite{EvGa}).

\begin{thm}\label{exuni}
Assume that the measure $\mu_0$ is absolutely continuous with respect to the Lebesgue measure $dx^n$. Then there is an optimal map $\varphi$ of (\ref{OT}) pushing $\mu_0$ forward to $\mu_1$ which is of the form
\begin{equation}\label{map}
\varphi(x)=x-\nabla f(x),
\end{equation}
where $f$ is a $c$-concave function. This map is unique up to a set of $\mu_0$ measure zero.

Moreover, if $\mu_1$ is also absolutely continuous with respect to $dx^n$, then the map
\[
\varphi^c(x):=x-\nabla f^c(x)
\]
is the optimal map pushing $\mu_1$ forward to $\mu_0$, where $f^c$ is the $c$-transform of $f$.
\end{thm}

Next, let us fix a positive constant $h$ and a positive function $\rho_0$ satisfying $\int_{\ntorus}\rho_0d\mu=1$ and consider the following minimization problem
\begin{equation}\label{JKOminre}
\inf\left[\frac{1}{2}\dis^2(\rho_0\mu,\rho\mu)+h\int_{\ntorus}(\log \rho-\log v_0+\Psi) \rho d\mu\right],
\end{equation}
where the infimum is taken over the set of $L^1$ functions $\rho:\ntorus\to [0,\infty)$ satisfying $\int_{\ntorus}\rho d\mu=1$.

\begin{thm}\label{JKO}
There is a unique minimizer $\rho$ of (\ref{JKOminre}) which is locally semi-convex on the set
\[
\{x|\rho_0(x)>0\}.
\]
and satisfying the following
\begin{equation}\label{eqn}
\varphi_*(\rho\mu)=\rho_0\mu,
\end{equation}
where
\[
\varphi(x)=x+h\nabla F(x)
\]
is the optimal map which pushes forward $\rho\mu$ to $\rho_0\mu$ and
\[
F:=\log \rho-\log v_0+\Psi.
\]
Moreover, if there is a $C^2$ positive solution of (\ref{eqn}), then it coincides with $\rho$.
\end{thm}

\begin{proof}
By the convexity of the functional in (\ref{JKOminre}), any minimizing sequence of (\ref{JKOminre}) converges weakly in $L^1$ to a unique minimizer $\rho$. Following \cite{JoKiOt}, we let $\psi_s$ be the flow of a vector field $X$ and let $(\psi_s)_*(\rho\mu)=\sigma_s\mu$. It follows that
\begin{equation}\label{d1}
\begin{split}
&\frac{d}{ds}\Big[\frac{1}{2}\int_{\ntorus}d^2(x,\psi_s(\bar\varphi(x)))\rho_0(x)+h\log\sigma_s(\psi_s(x))\rho(x)\\
&+h(-\log v_0(\psi_s(x))+\Psi(\psi_s(x)))\rho(x)d\mu(x)\Big]\Big|_{s=0}=0,
\end{split}
\end{equation}
where $\bar\varphi$ is the optimal map pushing $\rho_0\mu$ forward to $\rho\mu$.

Since $\sigma_s(\psi_s)v_0(\psi_s)\det(d\psi_s)=\rho v_0$, we have
\begin{equation}\label{d2}
\begin{split}
&\frac{d}{ds}\int_{\ntorus}(\log\sigma_s-\log v_0+\Psi)(\psi_s(x))\, \rho(x) d\mu(x)\Big]\Big|_{s=0}\\
&\quad=\int_{\ntorus} [-\di(X(x))+\left<X(x),\nabla(\Psi-2\log v_0)(x)\right>] \rho(x) d\mu(x).
\end{split}
\end{equation}

Using standard arguments in \cite{Mc}, it is not hard to see that
\begin{equation}\label{disest}
\begin{split}
&\frac{d}{ds}\left[\frac{1}{2}\int_{\ntorus}d^2(x,\psi_s(\bar\varphi(x)))\rho_0(x)d\mu(x)\right]\Big|_{s=0}\\
&=\frac{1}{2}\int_{\ntorus}\left<\nabla d^2_x,X\right>_{\bar\varphi(x)}\rho_0(x)d\mu(x).
\end{split}
\end{equation}

By combining this with (\ref{d1}) and (\ref{d2}), we have
\begin{equation}\label{der}
\begin{split}
&\int_{\ntorus}\di(X)\rho d\mu\\
&=\int_{\ntorus}\left<\frac{1}{2h}\nabla d^2_{\bar\varphi^{-1}(x)}+\nabla(\Psi-2\log v_0),X\right>_x\rho(x)d\mu(x).
\end{split}
\end{equation}

It follows that $\rho$ is Lipschitz and
\[
-h\nabla\log\rho(x)+h\nabla\log v_0-h\nabla\Psi(x)=\frac{1}{2}\nabla d^2_{\bar\varphi^{-1}(x)}(x)
\]
holds $\rho\mu$ almost everywhere. Therefore,
\[
x+h\nabla F(x)=\bar\varphi^{-1}(x)
\]
$\rho\mu$ almost everywhere. (\ref{eqn}) follows from this and Theorem \ref{exuni}.

It also follows from the above discussion that
\[
t\mapsto x+th\nabla F(x)
\]
is the unique minimizing geodesic between its endpoints for $\rho\mu$ almost all $x$. Therefore,
\[
h\nabla F(x)=-\nabla f^c(x)
\]
for $\rho\mu$ almost all $x$.

Since $c$-concave functions are locally semi-concave, $\log\rho$ is locally semi-convex on the open set $\{x\in M|\rho(x)>0\}$. It follows that the following equation holds Lebesgue almost everywhere on the same set
\begin{equation}\label{MAeqn}
\rho_0(\varphi(x))v_0(\varphi(x))\det(d\varphi(x))=\rho(x)v_0(x)
\end{equation}
where $\varphi(x)=x+h\nabla F(x)$. The rest follows from a simple application of maximum principle.

Let $\bar\rho$ be a positive $C^2$ solution of (\ref{MAeqn}). Let $g=\log\rho-\log\bar\rho$ and let $x'$ be the maximum point of $g$. Since $g$ is locally semi-convex, there is a sequence of points $x_i$ where $g$ is twice differentiable and the followings hold
\[
\lim_{i\to\infty}x_i=x',\quad \lim_{i\to\infty}\nabla g(x_i)=0,\quad \nabla^2 g(x_i)\leq \epsilon_i I
\]
for some sequence $\epsilon_i\to 0$ as $i\to\infty$. It follows that
\[
\begin{split}
&\rho(x_i)v_0(x_i)=\rho_0(\varphi(x_i))v_0(\varphi(x_i))\det(I+h\nabla^2F(x_i))\\
&\leq \rho_0(\varphi(x_i))v_0(\varphi(x_i))\det((1-\epsilon_i)I+h\nabla^2(\log\bar\rho-\log v_0+\Psi)(x_i)).
\end{split}
\]

By letting $i\to\infty$ and using (\ref{MAeqn}), we obtain $\rho\leq \bar\rho$ everywhere. On the other hand, $\int\rho v_0=\int\bar\rho v_0$. Therefore, $\rho\equiv \bar\rho$.
\end{proof}

\smallskip

\section{Regularity of Minimizers}

In this section, we show $C^2$ regularity and positivity of the minimizers in (\ref{JKOmin}). More precisely, we have the following result.

\begin{thm}\label{regularity}
Let
\begin{equation}\label{lk}
\lambda:=\sup_{\{w\in T\ntorus|\, |w|=1\}}\nabla^2 F(w,w).
\end{equation}
Assume that $\rho_0$ is a positive $C^{r,\alpha}$ function  for some $r\geq 2$ and $\alpha>0$. Assume also that $h\lambda\leq\frac{1}{8}$. Then there is a constant $h_0>0$ depending on $\Psi$, $v_0$, and $\rho_0$ such that the minimizer of (\ref{JKOminre}) with $h<h_0$ is a $C^{r+2,\alpha}$ solution of (\ref{MAeqn}).
\end{thm}

First, we prove the following a priori estimates for solutions of the equation (\ref{MAeqn}).

\begin{lem}\label{est}
Let $\rho$ be a $C^4$ positive solution of the equation (\ref{MAeqn}). If $I+h\nabla^2F\geq 0$, then
\begin{enumerate}
\item $(1+h||\nabla^2(\Psi-\log v_0)||_\infty)^n\sup_{\ntorus}(\rho_0v_0)\\
\geq\rho v_0\geq (1-h||\nabla^2(\Psi-\log v_0)||_{\infty})^n\inf_{\ntorus}(\rho_0v_0)$,
\item $(1-h||\nabla^2(\Psi-2\log v_0)||_{\infty})||\nabla F||_\infty\leq||\nabla F_0||_{\infty}$,
\item $0\leq h||\nabla^3(\Psi-2\log v_0)||_{\infty}||\nabla F||_{\infty}+\lambda_0\\
+(h^2||\nabla^3(\Psi-2\log v_0)||_\infty||\nabla F||_{\infty}\\
+2h||\nabla^2(\Psi-2\log v_0)||_{\infty}
+2h\lambda_0-1)\lambda\\
+(\frac{1}{3}h^3||\nabla^3(\Psi-2\log v_0)||_{\infty}||\nabla F||_{\infty}\\
+h^2||\nabla^2(\Psi-2\log v_0)||_{\infty}+h^2\lambda_0)\lambda^2$,
\end{enumerate}
where $|S|$ is the norm of the tensor $S$, $||S||_\infty=\sup_{\ntorus}|S|$,
\[
\lambda_0=\sup_{\{w\in T\ntorus|\, |w|=1\}}\nabla^2 F_0(w,w)
\]
and $F_0=\log \rho_0-\log v_0+\Psi$.
\end{lem}

\begin{proof}
The first assertion follows immediately from (\ref{MAeqn}) and the maximum principle.

For each fixed $x$ in the torus $\mathbb T^n$ and each vector $w$ in $\Real^n$, we let $\gamma(s)=x+sw$ and let $\varphi(x)=x+h\nabla F(x)$. It follows from (\ref{MAeqn}) that
\[
\begin{split}
&\log\rho_0(\varphi(\gamma(s)))+\log v_0(\varphi(\gamma(s)))+\log\det(I+h\nabla^2 F(\gamma(s)))\\
&=\log\rho(\gamma(s))+\log v_0(\gamma(s)).
\end{split}
\]
Note that the $v_0$ term in the above equation appears with a plus sign and in $F$ with a minus sign.

By differentiating this equation with respect to $s$, we obtain
\begin{equation}\label{D1}
\begin{split}
&\left<\nabla(\log\rho+\log v_0)(\gamma(s)),w\right>\\
&=\left<\nabla(\log\rho_0+\log v_0)(\varphi(\gamma(s))),(d\varphi)_{\gamma(s)}(w)\right>\\
&+h\tr((I+hS(s))^{-1}S'(s)),
\end{split}
\end{equation}
where $S(s)$ is the matrix defined by $S_{ij}(s)=\nabla^2F(\gamma(s))(\partial_{x_i},\partial_{x_j})$.

If $x$ is a point where $y\mapsto |\nabla F(y)|^2$ achieves its maximum, then $\nabla^2 F(x)(\nabla F(x))=0$ and $\nabla^3 F(x)(\nabla F(x),v,v)\leq -|\nabla^2 F(x)(v)|^2$ for any vector $v$.

By combining this with (\ref{D1}), we obtain
\[
\begin{split}
&\left<\nabla(\log\rho+\log v_0)(x),\nabla F(x)\right>-\left<\nabla(\log\rho_0+\log v_0)(\varphi(x)), d\varphi_x(\nabla F(x))\right>\\
&\leq-h\tr((I+hS(0))^{-1}S(0)^2)\leq 0.
\end{split}
\]

The above inequality together with $\nabla^2 F(x)(\nabla F(x))=0$ gives
\[
\begin{split}
&|\nabla F(x)|^2-\left<\nabla(\log\rho_0-\log v_0+\Psi)(\varphi(x)),d\varphi_x(\nabla F(x))\right>\\
&\leq \left<\nabla(-2\log v_0+\Psi)(x),\nabla F(x)\right>-\left<\nabla(-2\log v_0+\Psi)(\varphi(x)),d\varphi_x(\nabla F(x))\right>\\
&\leq h||\nabla^2(-2\log v_0+\Psi)||_\infty|\nabla F(x)|^2+h||\nabla(-2\log v_0+\Psi)||_\infty||\nabla^2F(x) (\nabla F(x))||\\
&= h||\nabla^2(-2\log v_0+\Psi)||_\infty|\nabla F(x)|^2.
\end{split}
\]

Therefore, by using $d\varphi_x(\nabla F(x))=\nabla F(x)+h\nabla^2 F(x)(\nabla F(x))=\nabla F(x)$, we obtain
\[
\begin{split}
&(1-h||\nabla^2(-2\log v_0+\Psi)||_\infty)|\nabla F(x)|^2\\
&\leq\left<\nabla(\log\rho_0-\log v_0+\Psi)(\varphi(x)),d\varphi_x(\nabla F(x))\right>\\
&\leq ||\nabla F_0||_\infty\,|\nabla F(x)|
\end{split}
\]
and the second assertion follows.

By differentiating (\ref{D1}), we obtain
\begin{equation}\label{D2}
\begin{split}
&\nabla^2(\log\rho+\log v_0)(x)(w,w)\\
&=\nabla^2(\log\rho_0+\log v_0)(\varphi(x))(d\varphi(w),d\varphi(w))\\
&+h\nabla^3F(x)(w,w,\nabla(\log\rho_0+\log v_0)(\varphi(x)))\\
&-h^2\tr(((I+hS(0))^{-1}S'(0))^2)+h\tr((I+hS(0))^{-1}S''(0)).
\end{split}
\end{equation}

If $(x,w)$ achieves the supremum in (\ref{lk}), then
\[
\begin{split}
&\nabla^3F(x)(w,w,v)=0,\\
&\left<S''(0)v,v\right>=\nabla^4F(x)(w,w,v,v)\leq 0
\end{split}
\]
for any vector $v$.

Therefore, by combining this with (\ref{D2}), we obtain
\[
\nabla^2(\log\rho+\log v_0)(x)(w,w)\leq\nabla^2(\log\rho_0+\log v_0)(\varphi(x))(d\varphi(w),d\varphi(w)).
\]
Since $d\varphi(w)=w+h\nabla^2F(x)(w)=w+h\lambda w$, it also follows that
\[
\begin{split}
&\lambda-(1+h\lambda)^2\lambda_0\\
&=\nabla^2(\Psi-2\log v_0)(x)(w,w)-\nabla^2(\Psi-2\log v_0)(\varphi(x))(d\varphi(w),d\varphi(w))\\
&\leq\int_0^1\frac{d}{dt}\nabla^2(\Psi-2\log v_0)(\Phi_t(x))(d\Phi_t(x)(w),d\Phi_t(x)(w))dt\\
&\leq h||\nabla^3(\Psi-2\log v_0)||_\infty||\nabla F||_\infty\left(1+h\lambda+\frac{1}{3}h^2\lambda^2\right)\\
 &+h||\nabla^2(\Psi-2\log v_0)||_\infty(2\lambda+h\lambda^2),
\end{split}
\]
where $\Phi_t(x)=tx+(1-t)\varphi(x)$.

The last assertion follows from this.
\end{proof}

\begin{proof}[Proof of Theorem \ref{regularity}]
Let $C$ be a constant such that
\[
\int_MCe^{\log v_0-\Psi}d\mu=1.
\]
First, if $\rho_0=Ce^{\log v_0-\Psi}$, then $\rho= Ce^{\log v_0-\Psi}$ is a smooth positive solution of (\ref{MAeqn}). For more general $\rho_0$, let us consider the following family of equations
\begin{equation}\label{MAinter}
\begin{split}
&\det(I+h\nabla^2F(x))=\frac{\rho(x)v_0(x)}{((Ce^{\log v_0-\Psi})^{1-s}(\rho_0)^sv_0)(x+h\nabla F(x))}.
\end{split}
\end{equation}

Let $s_0$ be the supremum over the set of all $s$ in $[0,1]$ for which (\ref{MAinter}) has a $C^4$ solution $\bar\rho_s$. By Theorem \ref{JKO}, $\bar\rho_s$ is a minimizer of (\ref{JKOminre}) with $\rho_0$ replaced by $(Ce^{\log v_0-\Psi})^{1-s}(\rho_0)^s$. It follows that $I+h\nabla^2(\log\bar\rho_s-\log v_0+\Psi)\geq 0$. Let us choose $h_0$ such that $1-h_0||\nabla^2(\Psi-2\log v_0)||_{\infty}>0$. Then, by the second statement of Lemma \ref{est}, the set of solutions $\{\log\bar\rho_s|0\leq s < s_0\}$ has a uniform $C^1$ bound depending on $\rho_0$.

Let $\lambda(s)=\sup_{x\in \ntorus, |w|=1}\nabla^2(\log\bar\rho_s-\log v_0+\Psi)(x)(w,w)$. Since
\[
\begin{split}
&\sup_{x\in \ntorus, |w|=1}\nabla^2[\log((Ce^{\log v_0-\Psi})^{1-s}(\rho_0)^s)-\log v_0+\Psi](x)(w,w)\\
&=s\sup_{x\in \ntorus, |w|=1}\nabla^2(\log\rho_0-\log v_0+\Psi)(x)(w,w)\\
&=s\lambda_0<\lambda_0,
\end{split}
\]
we can apply the third statement of Lemma \ref{est} and conclude that there are positive constants $h_0$ and $C$ depending only on $\rho_0$, $v_0$, and $\Psi$ such that
\[
\begin{split}
&0\leq h^2C+h\lambda_0+\left(hC+2h\lambda_0-1\right)h\lambda(s)+\left(hC+h\lambda_0\right)(h\lambda(s))^2\\
\end{split}
\]
for all $h<h_0$.

By assumption, we have $h\lambda_0\leq\frac{1}{8}$. Therefore, by choosing a smaller $h_0$,
\[
h^2C+h\lambda_0+\left(hC+2h\lambda_0-1\right)x+\left(hC+h\lambda_0\right)x^2
\]
has real root.

Since $\lambda(0)=0$, $\lambda(s)$ has a upper bounded independent of $s$.  Therefore, $\lambda(s)$ is bounded above independent of $s$ (again by assuming $h_0$ small enough).

This, together with (\ref{MAinter}), also gives a uniform positive lower bound of $I+h\nabla^2(\log\bar\rho_s-\log v_0+\Psi)$. Higher order estimates follow from standard elliptic theory \cite{GiTr}. Therefore, there is a solution $\bar\rho_{s_0}$ of (\ref{MAinter}) with $s=s_0$. By  elliptic theory and the implicit function theorem, we must have $s_0=1$.
\end{proof}

\smallskip

\section{$C^{0,\alpha}$-Convergence of the JKO Scheme}

In this section, we show the $C^{0,\alpha}$-convergence of the JKO scheme. For each fixed positive integer $N$, let $h=\frac{K}{N}$, where $K$ is a positive constant to be fixed later which depends only on $\rho_0$, $v_0$, $\Psi$, and dimension. Let $\rho_0:\mathbb T^n\to (0,\infty)$ be a smooth function. Then (\ref{JKOmin}) defines a sequence of functions $\rho_0^N:=\rho_0, \rho_1^N,...,\rho_N^N$. Let $F_k^N:=\log\rho_k^N-\log v_0+\Psi$. Then the followings are consequences of Lemma \ref{est}.

\begin{lem}\label{limest}
There are positive constants $C$ and $K$ depending only on $\rho_0$, $v_0$, $\Psi$, and dimension such that
\begin{enumerate}
\item $\frac{1}{C}\leq \rho_k^N\leq C$,
\item $||\nabla F_k^N||_\infty\leq C$,
\item $\frac{K\lambda_k^N}{N}\leq\frac{1}{8}$,
\end{enumerate}
for all positive integer $N$ and all $k=1,...,N$.
\end{lem}

\begin{proof}
We proceed by induction. Assume that the estimates hold and  $\frac{K}{N}\lambda_k\leq\frac{1}{8}$ for all $k\leq m-1$. Then, by Theorem \ref{regularity}, $\rho_m^N$ is smooth.
Assume that $K$ satisfies
\[
K<\max\left\{\frac{1}{||\nabla^2(\Psi-2\log v_0)||_\infty},\frac{1}{||\nabla^2(\Psi-\log v_0)||_\infty}\right\}.
\]
It follows from Lemma \ref{est} that
\[
\begin{split}
&e^{nC||\nabla^2(\Psi-\log v_0)||_\infty}\sup(\rho_0v_0)\\
&\geq \left(1+\frac{K}{N}||\nabla^2(\Psi-\log v_0)||_\infty\right)^{nN}\sup(\rho_0v_0)\\
&\geq\rho^N_m v_0\geq \left(1-\frac{K}{N}||\nabla^2(\Psi-\log v_0)||_\infty\right)^{nN}\inf(\rho_0v_0)\\
&\geq e^{-nK||\nabla^2(\Psi-\log v_0)||_\infty}\inf(\rho_0v_0)>0
\end{split}
\]
and
\[
\begin{split}
||\nabla F_m^N||_\infty &\leq\frac{1}{\left(1-\frac{K}{N}||\nabla^2(\Psi-2\log v_0)||_\infty\right)^N}||\nabla F_0^N||_\infty\\
 &\leq e^{-K||\nabla^2(\Psi-2\log v_0)||_\infty}||\nabla F_0||_\infty.
\end{split}
\]
This proves the first two assertions.

By the above estimates and Lemma \ref{est}, there are positive constants $C$ and $h_0$ depending only on $\Psi$, $v_0$, and $\rho_0$ such that
\[
\begin{split}
&0\leq h^2C+h\lambda^N_{k-1}+(hC+2h\lambda^N_{k-1}-1)h\lambda^N_k+\Big(hC+h\lambda^N_{k-1}\Big)(h\lambda^N_k)^2
\end{split}
\]
for all $h<h_0$.

Assume that $K\leq \frac{1}{6C}$. It follows as in the proof of Theorem \ref{regularity} that
\[
\begin{split}
h\lambda_k^N&\leq\frac{2(h^2C+h\lambda_{k-1}^N)}{
1-hC-2h\lambda_{k-1}^N+\sqrt{1+h^2C^2-2hC-4h^3C^2-4h\lambda_{k-1}^N-4h^3C\lambda_{k-1}^N}}\\
&\leq\frac{2(h^2C+h\lambda_{k-1}^N)}{
1-hC-2h\lambda_{k-1}^N+\sqrt{1-3hC-4h\lambda_{k-1}^N}}\\
&\leq\frac{h^2C+h\lambda_{k-1}^N}{1-2hC-3h\lambda_{k-1}^N}.
\end{split}
\]
for all $k\leq m-1$. Note that $1-3hC-4h\lambda_{k-1}^N>0$ since $K\leq \frac{1}{6C}$ and $\frac{K}{N}\lambda_{k-1}=h\lambda_{k-1}\leq\frac{1}{8}$.

By assuming that the constant $C$ is large enough, we have
\[
\begin{split}
&\frac{ph^2C+\frac{h^2C+h\lambda_{k}^N}{1-2hC-3h\lambda_{k}^N}}{1-2phC-3p\frac{h^2C+h\lambda_{k}^N}{1-2hC-3h\lambda_{k}^N}}\\
&=\frac{ph^2C(1-2hC-3h\lambda^N_{k})+h^2C+h\lambda^N_{k}}{\left(1-2phC\right)(1-2hC-3h\lambda^N_{k})-3ph^2C-3ph\lambda^N_{k}}\\
&=\frac{(p+1)h^2C-2ph^3C^2-3ph^3C\lambda^N_{k}+h\lambda^N_{k}}{1-2\left(p+1\right)hC-3(p+1)h\lambda^N_{k}+4ph^2C^2+6ph^2C\lambda^N_{k}-3ph^2C}\\
&\leq\frac{(p+1)h^2C+h\lambda^N_{k}}{1-2\left(p+1\right)hC-3(p+1)h\lambda^N_{k}}.
\end{split}
\]

By iterating the two inequalities above, we obtain
\[
\begin{split}
\frac{K\lambda_m^N}{N}&\leq \frac{K}{N}\frac{\lambda_0+\frac{Km}{N}C}{1-\frac{2CKm}{N}-\frac{3Km\lambda_0}{N}}\\
&\leq \frac{K}{N}\frac{\lambda_0+KC}{1-2CK-3K\lambda_0},
\end{split}
\]
where $\lambda_0=\sup_{x\in M, |w|=1}\nabla^2F_0(x)(w,w)\leq C$ and $F_0=\log\rho_0-\log v_0+\Psi$.

Assume that the constant $K$ satisfies $\frac{K(\lambda_0+KC)}{1-2CK-3K\lambda_0}<\frac{1}{8}$ and $K<\frac{1}{2C+3\lambda_0}$. Then the last assertion follows.
\end{proof}

Finally, we finish the proof of Theorem \ref{main}. The arguments are mild modifications of the ones in \cite{JoKiOt,Ot,AmGiSa} combined with the estimates in Lemma \ref{limest}.

\begin{proof}
Let $\phi^N_t:\left[0,K\right]\times M\to\Real$ be defined by
\[
\phi^{N}_t=\rho_{k}^N
\]
if $t$ is in $\left[\frac{kK}{N},\frac{(k+1)K}{N}\right)$ and $k=0,...,N-1$.

By (\ref{JKOmin}),
\[
\begin{split}
&\frac{1}{2}\dis^2(\rho_{k}^N\mu,\rho_{k+1}^N \mu)+\frac{K}{N}\int_M F_{k+1}^Nd\mu\leq \frac{K}{N}\int_M F_k^N d\mu.
\end{split}
\]
Therefore,
\begin{equation}\label{smd}
\begin{split}
&\frac{1}{2}\sum_{k=0}^{N-1}\dis^2(\rho_k^N\mu,\rho_{k+1}^N \mu)\\
&\leq \frac{K}{N}\left(\int_M F_0^N \rho_0 d\mu-\inf\int_M(\log \rho-\log v_0+\Psi) \rho d\mu\right),
\end{split}
\end{equation}
where the infimum is taken over all $\rho:\mathbb T^n\to[0,\infty)$ satisfying
\[
\int_{\mathbb T^n}\rho\, d\mu=1.
\]

Therefore, we obtain
\[
\begin{split}
&\int_0^K\int_{\mathbb T^n}|\nabla\log\phi_t^N(x)-\nabla\log v_0+\nabla\Psi(x)|^2\phi_t^N(x)d\mu dt\\
&=\sum_{k=0}^{N-1}\dis^2(\rho_k\mu,\rho_{k+1}\mu)\leq C.
\end{split}
\]

Hence, $t\mapsto \sqrt{\int_{\mathbb T^n}|\nabla\log\phi_t^N(x)-\nabla\log v_0+\nabla\Psi(x)|^2\phi_t^N(x)d\mu}$ converges weakly in $L^2$ to a function $A:[0,K]\to\Real$. On the other hand, we have
\[
\begin{split}
&\dis(\phi_{\tau_0}^N\mu,\phi_{\tau_1}^N \mu)\leq \sum_{k=m_0}^{m_1-1}\dis(\rho_{k-1}^N\mu,\rho_k^N \mu)\\
&=\int_{\frac{(m_0-1)K}{N}}^{\frac{(m_1-1)K}{N}}\sqrt{\int_{\mathbb T^n}|\nabla\log\phi_t^N(x)-\nabla\log v_0+\nabla\Psi(x)|^2\phi_t^N(x)d\mu}\,dt,
\end{split}
\]
where $\frac{(m_i-1)K}{N}\leq\tau_i<\frac{m_iK}{N}$, $i=0, 1$.

By letting $N\to\infty$, we obtain
\begin{equation}\label{equi}
\begin{split}
&\liminf_{N\to\infty}\dis(\phi_{\tau_0}^N\mu,\phi_{\tau_1}^N \mu)\leq \int_{\tau_0}^{\tau_1}A(t) dt.
\end{split}
\end{equation}

Let $\mathcal D$ be a dense subset of $[0,K]$. By Lemma \ref{limest} and a diagonal argument, there is a subsequence $\phi_t^{N_k}$ which converges in $C^{0,\alpha}(\mathbb T^n)$ to a function $\phi_t$ for all $t$ in $\mathcal D$. By Lemma \ref{limest} and (\ref{equi}), we can extend $\{\phi_t|t\in\mathcal D\}$ to a unique curve $\{\phi_t|t\in[0,K]\}$ contained in $C^{0,\alpha}(\mathbb T^n)$ by continuity. Next, we show that $\phi_t^{N_k}$ converges to $\phi_t$ in $C^{0,\alpha}(\ntorus)$ for all $t$ in $[0,K]$. By Lemma \ref{limest}, it is enough to show that any convergence subsequence (say $\phi_t^{N_{k_m}}$) converges to $\phi_t$ in $C^{0,\alpha}(\ntorus)$. Suppose that the uniform limit of $\phi_t^{N_{k_m}}$ is $\tilde\phi_t$. Then, by (\ref{equi}),
\[
\begin{split}
\dis(\phi_s\mu,\tilde\phi_{t}\mu)= \lim_{m\to\infty}\dis(\phi^{N_{k_m}}_s\mu,\phi_{t}^{N_{k_m}}\mu)\leq \int_s^tA(\tau) d\tau
\end{split}
\]
for all $s$ in $\mathcal D$. It follows from this and the definition of $\phi_t$ that $\phi_t=\tilde \phi_t$. We also know that $\phi_t$ is Lipschitz in space.

Finally, if $\Phi_k^N(x)=x+\frac{K}{N}\nabla F_k^N(x)$ and $\xi$ is any function on $\Real\times\ntorus$ with compact support, then we have
\[
\begin{split}
&\left|\int_{\mathbb T^n}\xi_t(\rho_{k}^N-\rho_{k-1}^N)d\mu+\frac{K}{N}\int_{\mathbb T^n}\left<\nabla\xi_t,\nabla F_k^N\right>\rho_k^Nd\mu\right| \\
&\leq \frac{K^2}{2N^2}\int_{\mathbb T^n}\left|\left<\nabla^2\xi_t(\nabla F_k^N),\nabla F_k^N\right>\right|\rho_k^Nd\mu\\
&\leq \frac{1}{2}||\nabla^2\xi_t||_{C^0}\frac{K^2}{N^2}\int_{\mathbb T^n}\left|\nabla F_k^N\right|^2\rho_k^Nd\mu.
\end{split}
\]

On the other hand, by (\ref{d1}), (\ref{d2}), and (\ref{disest}), we have
\[
\begin{split}
&\int_M [-\di(X)+\left<X,\nabla(\Psi-2\log v_0)\right>-\left<\nabla F_k^N,X\right>] \rho_k d\mu=0.
\end{split}
\]
Therefore, by choosing $X=\nabla \xi_t$, we obtain
\[
\begin{split}
&\Big|\int_{\mathbb T^n}[-\Delta\xi_t+\left<\nabla\xi_t,\nabla(\Psi-2\log v_0)\right>]\rho_k^Nd\mu \\
&+\int_{\mathbb T^n}\frac{N}{K}\xi_t(\rho_{k}^N-\rho_{k-1}^N)d\mu\Big|\leq \frac{1}{2}||\nabla^2\xi||_{C^0}\frac{K}{N}\int_{\mathbb T^n}\left|\nabla F_k^N\right|^2\rho_k^Nd\mu.
\end{split}
\]

For each fix time $T$ in $[0,K]$, let $k_N$ be the integer satisfying $\frac{(k_N-1)K}{N}\leq T<\frac{k_NK}{N}$. By applying (\ref{smd}), we obtain
\[
\begin{split}
&\sum_{k=1}^{k_N}\int_{\mathbb T^n}\xi_{\frac{(k-1)K}{N}}\rho_{k}^N-\xi_{\frac{(k-1)K}{N}}\rho_{k-1}^Nd\mu\\
&\to\int_0^T\int_{\mathbb T^n}[\Delta\xi_t-\left<\nabla\xi_t,\nabla(\Psi-2\log v_0)\right>]\phi_td\mu\, dt
\end{split}
\]
as $N\to\infty$.

On the other hand, we have
\[
\begin{split}
&\sum_{k=1}^{k_N}\int_{\mathbb T^n}\xi_{\frac{(k-1)K}{N}}\rho_{k}^N-\xi_{\frac{(k-1)K}{N}}\rho_{k-1}^Nd\mu\\
&\to\int_{\mathbb T^n}\xi_T\phi_Td\mu-\int_{\mathbb T^n}\xi_0\phi_0d\mu-\int_0^T\int_{\mathbb T^n}\partial_t \xi_t \,\phi_td\mu\, dt
\end{split}
\]
as $N\to\infty$.

By combining the above discussions with (\ref{v0}), we see that $\phi$ is a weak solution (and hence the unique classical solution, see \cite{Li}) to the equation 
\begin{equation}\label{para}
\dot \phi=\Delta \phi+\left<\nabla\phi,\nabla\Psi\right>+f\phi
 \end{equation}
on $[0,K]\times\ntorus$. This gives the result for a short time. The result for long time follows from $C^2$ estimate for linear parabolic equations (see \cite{Li}).

Finally, it is not hard to see that the whole sequence $\phi_t^N$ converges to $\phi_t$ for each $t$. Indeed, suppose there is a time $t_0$ and a subsequence $\phi_{t_0}^{N_i}$ converging in $C^{0,\alpha}(\ntorus)$ to $\bar\phi_{t_0}$ which is different from $\phi_{t_0}$. By the same procedure as above, we can find a subsequence $\phi_{t}^{N_{i_k}}$ of $\phi_{t}^{N_i}$ which converges uniformly to a weak solution of (\ref{para}). By the uniqueness of solution, this limit is $\phi_t$. In particular, $\bar\phi_{t_0}$ and $\phi_{t_0}$ should be the same. 
\end{proof}

\end{document}